\documentclass{article}

\usepackage{prelude}

\newcommand{\A}{\mathcal A}
\newcommand{\St}{\mathfrak S}
\renewcommand{\L}{\mathfrak L}

\declaretheorem[name = Example, style = example, numberwithin=section]{example}
\declaretheorem[name = Theorem, style = theorem, numberwithin=section]{theorem}
\declaretheorem[name = Proposition, style = theorem, numberwithin=section]{proposition}
\declaretheorem[name = Corollary, style = theorem, numberwithin=section]{corollary}
\declaretheorem[name = Lemma, style = theorem, numberwithin=section]{lemma}
\declaretheorem[name = Proof, style = proof, numbered=no]{proof}
\declaretheorem[name = Remark, style = remark]{remark}

\title{Old and new powerful tools for the normal ordering problem and noncommutative binomials}
\author{Kei Beauduin}
\date{}

\begin{document}

\maketitle

\begin{abstract}
    In this paper, we derive formal general formulas for noncommutative exponentiation and the exponential function, while also revisiting an unrecognized, and yet powerful theorem. These tools are subsequently applied to derive counterparts for the exponential identity $e^{A+B} = e^A e^B$ and the binomial theorem $(A+B)^n = \sum \binom{n}{k} A^k B^{n-k}$ when the commutator $[B, A]$ is either an arbitrary quadratic polynomial or a monomial in $A$ or $B$. Analogous formulas are found when the commutator is bivariate. Furthermore, we introduce a novel operator bridging between the normal and antinormal ordered forms.
\end{abstract}

\section{Introduction}

Let $\A$ be an associative $\K$-algebra (the field $\K$ is typically either $\R$ or $\C$). We define \emph{normal ordering} as the process of arranging a polynomial or a formal series in terms of $A$ and $B\in\A$, such that all occurrences of $A$ appear to the left of products of $B$. If such an arrangement exists, the normal ordered form is unique. When the $A$'s appear to the right of products with $B$, we will refer to as \emph{antinormal ordering}. Both of these processes are trivial when $A$ and $B$ \emph{commute}, as, for example, for the power of a product, we simply have $(AB)^n = A^n B^n = B^n A^n$. A less trivial but iconic normal ordering result in the case of commutativity is the normal ordering of a binomial $(A+B)^n$, also known as \emph{Newton's binomial theorem}
\[
(A+B)^n = \sum_{k=0}^n \binom{n}{k} A^k B^{n-k}.
\]
where $\binom{n}{k} \defeq \frac{n!}{k!(n-k)!}$ is the \emph{binomial coefficient}. However, the process of normal ordering when the variables do not commute is known to be a challenging subject, as the findings are often only partial results. Numerous contributions addressing this topic are scattered throughout the literature (see, e.g., \cite{Sack,Berry,Benaoum98,Levandovskyy,GenStir,Zassclosed,Ore} and others can be found in surveys such as \cite{book,survey,Blasiak}).

Normal ordering is predominantly studied in the context of univariate commutators, as the existence of a normal ordered form may be uncertain in the bivariate case (as discussed in \Cref{s:bivariate}), yet a general univariate case, encompassing both binomials and powers of a product, was completely solved by Viskov \cite{Viskov97}. More specifically, he solved the equivalent problem of providing an expression to the ordered form of the exponential of a bivariate term. For the sum and the product of two commuting variables, his formula reduces to the well-known \emph{exponential identities}:
\begin{equation}\label{e:EIs}
    e^{A+B} = e^A e^B, \qquad  e^{AB} = (e^A)^B,
\end{equation}
where the exponential is defined by $e^A = \exp(A) \defeq \sum_{n\geq0}A^n/n!$. This result was only recognized decades later by Mansour and Schork in \cite{book}, albeit focusing solely on the binomial case. For that reason, in \Cref{s:Viskov} we aim to provide a comprehensive exposition of Viskov's results and demonstrate in \Cref{s:monomial} an application where the \emph{commutator} $[B, A] \defeq BA - AB$ is of the form $hA^s$.

Before delving into Viskov's theorem and in order to formalize the second exponential identity, we will introduce a formal \emph{exponentiation} $A^B$ within noncommutative algebras, where both $A$ and $B$ are elements of $\A$. Then, utilizing \emph{Stirling numbers}, we will establish its basic properties and explore its interaction with the exponential function (detailed in \Cref{s:expo,BCH}). These new tools will also enable us to establish the exponential identity under the constraint that $[B, A] = \alpha + \epsilon A - \lambda A^2$ (see \Cref{s:quadra}), where $\alpha \in \K$ is assimilated as an element of $\A$ \footnote{If $\bm1_\A$ is the neutral element according to the product of the algebra $\A$, then it reasonable to define $\alpha \defeq \alpha \cdot \bm1_\A \in \A$ as it is compatible with the use of any power series $f$: $f(x) = f(x\cdot\bm1_\A) = f(x)\cdot\bm1_\A$. In particular $1 = \bm1_\A$.}.

Additionally, in \Cref{s:commutator} we will construct a new operator $\L$ that addresses a long-standing need, facilitating the transition between a normal ordered expression and its antinormal counterpart.

\section{Preliminaries}

\subsection{Stirling numbers}

We will denote the \emph{falling factorial} by the Pochhammer symbol $(x)_n \defeq x (x-1) \ldots (x-n+1)$ and define the \emph{Stirling numbers of the first kind} $\coeff{n}{k}$ (and its signed version $s(n,k)$) by
\begin{equation}\label{e:stir1}
    (x)_n = \sum_{k=0}^n \coeff{n}{k} (-1)^{n-k} x^k = \sum_{k=0}^n s(n,k) x^k,
\end{equation}
and the \emph{Stirling number of the second kind} $\Stir{n}{k}$ by \cite[Chapter 26.8]{NIST}
\begin{equation}\label{e:stir2}
    x^n = \sum_{k=0}^n \Stir{n}{k} (x)_k.
\end{equation}
Since these numbers are defined as coefficients of an expression on some basis, if $k$ is an integer different from $0, \ldots, n$, then $\coeff{n}{k} = \Stir{n}{k} = 0$.
The well-known Stirling generating functions are
\begin{equation}\label{e:stir1gf}
    \frac{(-\log(1-x))^k}{k!} = \sum_{n=k}^\infty \coeff{n}{k} \frac{x^n}{n!},
\end{equation}
and
\begin{equation}\label{e:stir2gf}
    \frac{(e^x-1)^k}{k!} = \sum_{n=k}^\infty \Stir{n}{k} \frac{x^n}{n!}.
\end{equation}
Furthermore, the closed-form formula for the Stirling numbers of the second kind is as follows:
\begin{equation}\label{e:stirformula}
    \Stir{n}{k} = \inv{k!} \sum_{i=0}^k \binom{k}{i} (-1)^{k-i} i^n.
\end{equation}

\subsection{Exponentiation}\label{s:expo}

By generalizing the binomial coefficient
\begin{equation}
    \binom{x}{n} \defeq \frac{(x)_n}{n!},
\end{equation}
we can define, for $A, B \in \A$, the following formal power series
\begin{equation}\label{e:expo}
    A^B \defeq \sum_{n=0}^\infty (A-1)^n \binom{B}{n}.
\end{equation}
Since we did not require the algebra to be commutative, in general, the form in which the power series is usually written is not true
\[
A^B \neq \sum_{n=0}^\infty \binom{B}{n} (A-1)^n.
\]
The convention we have adopted may not be ideal in all situations however, it offers the advantage of readability, ensuring that the reading order of $A$ and $B$ in $A^B$ aligns with the product order of its formula.

\begin{remark}
    For most elements of most nonformal algebras, this series diverges (similar to divergence in $\C$). It should also diverge for $B = x$ when $x$ is negative and $A$ is not invertible in $\A$. However, any representation of exponentiation that does converge in $\A$ should satisfy \cref{e:expo} as a valid formula when convergence occurs. Consequently, these representations should also exhibit the same properties as the one we will demonstrate below.
\end{remark}

\begin{proposition}\label{p:expo}
    For $A, B, C \in \A$  satisfying $[A, B] = [B, C] = 0$
    \begin{itemize}
        \item $A^B A^C = A^{B+C}$
        \item $(A^B)^C = A^{BC}$.
    \end{itemize}
\end{proposition}
\begin{proof}
    We are able to perform the Cauchy product because $A$ commutes with $B$ and Chu-Vandermonde's convolution because $B$ commutes with $C$
    \[
    A^B A^C = \sum_{n=0}^\infty (A-1)^n \sum_{k=0}^n \binom{B}{k} \binom{C}{n-k} = \sum_{n=0}^\infty (A-1)^n \binom{B+C}{n} = A^{B+C}.
    \]
    Now, it is easy to deduce that for positive integers $n$
    \[
    (A^B)^n = A^B A^B \cdots A^B = A^{B + B} \cdots A^B = \ldots = A^{nB},
    \]
    which we can use to fully prove the second point
    \begin{align*}
        (A^B)^C
        &= \sum_{m=0}^\infty (A^B-1)^m \binom{C}{m} = \sum_{m=0}^\infty \sum_{k=0}^m \binom{m}{k} (-1)^{m-k} (A^B)^k \binom{C}{m} \\
       &= \sum_{m=0}^\infty \sum_{k=0}^m \binom{m}{k} (-1)^{m-k} A^{kB} \binom{C}{m} \\
        &= \sum_{m=0}^\infty \sum_{k=0}^m \binom{m}{k} (-1)^{m-k} \sum_{n=0}^\infty (A-1)^n \binom{kB}{n} \binom{C}{m},
    \end{align*}
    then we use \cref{e:stir1}
    \begin{align*}
        (A^B)^C
        &= \sum_{n=0}^\infty \frac{(A-1)^n}{n!} \sum_{m=0}^\infty \sum_{k=0}^m \binom{m}{k} (-1)^{m-k}  \bra{\sum_{j=0}^n s(n,j) (kB)^j} \binom{C}{m} \\
        &= \sum_{n=0}^\infty \frac{(A-1)^n}{n!} \sum_{m=0}^\infty  \sum_{j=0}^n s(n,j) B^j \binom{C}{m} \sum_{k=0}^m \binom{m}{k} (-1)^{m-k} k^j.
    \end{align*}
    In the inner sum, \cref{e:stirformula} and (\ref{e:stir2}) lead to
    \begin{align*}
        \sum_{m=0}^\infty \frac{(C)_m}{m!} \sum_{k=0}^m \binom{m}{k} (-1)^{m-k} k^j = \sum_{n=0}^j (C)_m \Stir{j}{m} = C^j.
    \end{align*}
    The commutativity of $B$ and $C$ allows us to write $B^j C^j = (BC)^j$. Finally, applying \cref{e:stir1} once again, we arrive at
    \[
    (A^B)^C = \sum_{n=0}^\infty \frac{(A-1)^n}{n!} \sum_{j=0}^n s(n,j) B^j C^j = \sum_{n=0}^\infty \frac{(A-1)^n}{n!} (BC)_n = A^{BC}.
    \]
\end{proof}

The following proposition can be seen as a kind of notation trick, facilitated by our definition of exponentiation, to express this infinite sum more succinctly. It also enables us to utilize the properties we have just proven.

\begin{proposition}\label{p:e^A^B}
    For $A, B\in\A$
    \[
    \sum_{n=0}^\infty \inv{n!} A^n B^n = (e^A)^B.
    \]
\end{proposition}

\begin{proof}
    We make use of \cref{e:stir2gf}, exchange the sums and conclude with \cref{e:stir2}
    \begin{align*}
        (e^A)^B
        &= \sum_{k=0}^\infty (e^A - 1)^k \binom{B}{k} = \sum_{k=0}^\infty \sum_{n=k}^\infty \Stir{n}{k} \frac{A^n}{n!} (B)_k \\
        &= \sum_{n=0}^\infty \frac{A^n}{n!} \sum_{k=0}^n \Stir{n}{k} (B)_k = \sum_{n=0}^\infty \frac{A^n}{n!} B^n.
    \end{align*}
\end{proof}

\begin{example}\label{x:shift}
    In the algebra of linear transformations over analytic functions, the most important operators are $\x$ and $D$, where $\x f(x) \defeq xf(x)$ and $D f(x) \defeq f'(x)$, where $f'$ denotes the derivative of $f$. The product of the algebra is composition, and the exponents thus represent repeated self-composition. Using \Cref{p:e^A^B} and Taylor's formula, we can express the generalization of the \emph{shift} map $f(x) \mapsto f(x + g(x))$ in this manner:
    \begin{equation}\label{e:shift}
        f(x + g(x)) = \sum_{n=0}^\infty \frac{g(x)^n}{n!} f^{(n)}(x) = \sum_{n=0}^\infty \frac{g(\x)^n}{n!} D^n f(x) = (e^{g(\x)})^D f(x),
    \end{equation}
    which appears to be new. \Cref{e:shift} can also be seen as a \emph{composition operator}, indeed
    \begin{equation}
        (e^{g(\x) - \x})^D f(x) = f(g(x)).
    \end{equation}
\end{example}

\subsection{The commutator and the \texorpdfstring{$\L$}{L} operator}\label{s:commutator}

The commutator $[B, A] = BA - AB$ is a bilinear anticommutative operation that indicates the failure of commutativity between two elements of $\A$. These two properties combined allow for many manipulations that we will implicitly use throughout:
\[
[A, B] = [A, A+B] = -[B, A] = [-B, A] = [B, -A].
\]
The commutator yields the following simple identity
\begin{equation}\label{e:comid}
[B, A] = (A+B)^2 - A^2 - 2AB - B^2,
\end{equation}
due to the noncommutative expansion $(A+B)^2 = A^2 + AB + BA + B^2$. The following lemma, reminiscent of the chain rule, is a classical result in normal ordering.

\begin{lemma}\label{l:chain}
    If $f$ is a power series and $[[B, A], A] = 0$ then
    \[
    [B, f(A)] = [B, A] f'(A).
    \]
\end{lemma}

\begin{proof}
    First, we prove the result for monomials, and we see that the formula holds for $f(A) = 1$ and $f(A) = A$. We suppose for $n > 0$
    \[
    B A^n = A^n B + n [B, A] A^{n-1},
    \]
    now we multiply by $A$ to the right
    \begin{align*}
        B A^{n+1} &= A^n B A + n [B, A] A^n \\
        &= A^n (AB + [B, A]) + n [B, A] A^n \\
        &= A^{n+1} B + (n+1)[B, A]A^n,
    \end{align*}
    that is, $[B, A^{n+1}] = (n+1)[B, A]A^n$. Due to the linearity of all operations, the extension to power series is easily obtained.
\end{proof}

The \emph{adjoint} map is a derivation simply defined by $\ad_B(A) \defeq [B, A]$. With this definition in place, we can express the analogue of the Leibniz rule below, along with its antinormal form.

\begin{lemma}\label{l:leibniz}
    If $A, B \in \A$ then for all natural integers $n$
    \[
    B^n A = \sum_{k=0}^n \binom{n}{k} \ad_B^k(A) B^{n-k} \quad\et\quad A B^n = \sum_{k=0}^n \binom{n}{k} B^{n-k} (-\ad_B)^k(A).
    \]
\end{lemma}

The proofs are omitted because they can readily be accomplished by induction. We are now prepared to demonstrate the main theorem of this section.

\begin{theorem}\label{t:L}
    Let $p$ be a power series and $\A, \A'$ be two algebras generated by $A, B$ and $C, D$, respectively, satisfying $[B, A] = p(A)$ and $[D, C] = p(D)$. Let $X \in \A$ be uniquely written in normal-ordered form as
    \[
    X = \sum_{i, j} x_{i, j} A^i B^j.
    \]
    If we define the map $\L : \A \to \A'$ by
    \[
    \L(X) \defeq \sum_{i,j} x_{i,j} C^j D^i,
    \]
    then
    \begin{itemize}
        \item $\L$ is linear
        \item For $X, Y\in\A$, $\L(XY) = \L(Y) \L(X)$
        \item For a power series $f$, $\L(f(X)) = f(\L(X))$.
    \end{itemize}
\end{theorem}

This theorem proves particularly valuable in the context of obtaining the antinormal form for normal-ordered expressions. We will illustrate this significance in \Cref{s:Viskov}.

\begin{proof}
    The linearity property follows from the unique decomposition into the normal ordered form. To establish the other property, we first express $Y$ in normal ordered form as $Y = \sum_{k, \ell} y_{k, \ell} A^k B^\ell$. Next, we replicate this process for $XY$ using \Cref{l:leibniz}:
    \[
    XY = \sum_{i,j,k,\ell} x_{i,j} y_{k,\ell} A^i B^j A^k B^\ell = \sum_{i,j,k,\ell} x_{i,j} y_{k,\ell} A^i \sum_{r=0}^j \binom{j}{r} \ad_B^r(A^k) B^{j-r} B^\ell.
    \]
    By repeatedly applying \Cref{l:chain}, we deduce that $\ad_B^r(A^k)$ is a power series in $A$, confirming that the expression above is indeed a normal ordered form. With this understanding and as we want to apply $\L$, we now focus on elucidating the term $\L(A^i \ad_B^r(A^k) B^{j-r+\ell})$. Given that $[-C, D] = p(D)$ is analogous to $[B, A] = p(A)$, it is straightforward to see that this term simplifies to $C^{j-r+\ell} D^i \ad_{-C}^r(D^k)$. Thus, after applying $\L$, we arrive at
    \[
    \L(XY) = \sum_{i,j,k,\ell} x_{i,j} y_{k,\ell} \sum_{r=0}^j \binom{j}{r} C^{j-r+\ell} D^i \ad_{-C}^r(D^k).
    \]
    The above can be simplified using the simple fact that $\ad_{-C} = -\ad_C$ and the second formula of \Cref{l:leibniz}:
    \begin{align*}
        \L(XY)
        &= \sum_{i,j,k,\ell} x_{i,j} y_{k,\ell} C^\ell \sum_{r=0}^j \binom{j}{r} C^{j-r}  (-\ad_C)^r(D^k) D^i \\
        &= \sum_{i,j,k,\ell} x_{i,j} y_{k,\ell} C^\ell D^k C^j D^i = \bra{\sum_{k,\ell}  y_{k,\ell} C^\ell D^k} \bra{\sum_{i,j} x_{i,j} C^j D^i} \\
        &= \L(Y) \L(X),
    \end{align*}
    as desired. Inducting on $n$, we then derive $\L(X^n) = \L(X)^n$, which readily extends to power series due to the linearity of $\L$.
\end{proof}

\begin{example}
    In the very special case where the power series $p$ reduces to a constant $c$, we can simplify our choice for $C$ and $D$ by setting them to $A$ and $B$ respectively. The validity of the choice is ensured by $[B, A] = p(A) = p(B) = c$. We can opt for $A = \x$ and $B = D$ (where, here, $D$ is the derivative) since, in this context, $c = 1$. This choice leads to the intriguing formula for the composition operator of \Cref{x:shift}:
    \[
    \L((e^{g(\x)-\x})^D) = \sum_{n=0}^\infty \inv{n!} \L((g(\x) - \x)^n D^n) = \sum_{n=0}^\infty \inv{n!} \x^n (g(D) - D)^n = (e^\x)^{g(D) - D},
    \]
    which, according to \cite[Proposition 21]{XD}, is precisely an \emph{umbral operator} when $g$ represents power series with a compositional inverse.
\end{example}

\subsection{Implications of the Baker-Campbell-Hausdorff formula}\label{BCH}

The \emph{Baker-Campbell-Hausdorff formula} (BCH) is the solution for $Z$ of the equation $e^X e^Y = e^Z$ for non necessarily commuting $X, Y \in \A$, which takes the form of a (formal) \emph{Lie series}, i.e., a series in iterated commutators of $X$ and $Y$. Later discovered, \emph{Dynkin's formula} (BCHD) explicitly expresses the BCH. A clear exposition and proof of the BCHD are given in \cite{BCHD}.
\begin{theorem}[Baker-Campbell-Hausdorff-Dynkin formula]\label{t:BCHD}
    For $X, Y \in \A$,
    \[
    \log(e^X e^Y) = \sum_{k=1}^\infty \frac{(-1)^{k-1}}{k} \sum_{\substack{m_1+n_1>0 \\[-4pt] \vdots \\ m_k + n_k > 0}} \frac{\overbrace{[X,[\ldots, [X}^{m_1}, \overbrace{[Y,[ \ldots, [Y}^{n_1}, \overset{\substack{\dots\\[8pt]}}{[\ldots} \overbrace{[X, [\ldots, [X}^{m_k}, \overbrace{[Y,[ \ldots, [Y}^{n_k}, [\ldots]\ldots]}{\bra{\sum_{i=1}^k (m_i + n_i)} \bra{\prod_{i=1}^k m_i! \cdot n_i!}},
    \]
    with the understanding that $[Y] \defeq Y$.
\end{theorem}

The reader should keep in mind that many terms in the expansion are zeros, for example, when the term $[Y, Y]$ appears.

\begin{example}\label{x:BCHD}
    The first few terms are
    \begin{align*}
        \log(e^X e^Y) &= X + Y + \inv2 [X, Y] + \inv{12} ([X, [X, Y]] + [Y [Y, X]]) - \inv{24} [Y, [X, [X, Y]]] \\
        &- \inv{720} ([Y, [Y, [Y, [Y, X]]]] + [X, [X, [X, [X, Y]]]])\\
        &+ \inv{360} ([X, [Y, [Y, [Y, X]]]] + [Y, [X, [X, [X, Y]]]]) \\
        &+ \inv{120} ([Y, [X, [Y, [X, Y]]]] + [X, [Y, [X, [Y, X]]]]) \\
        &+ \ldots 
    \end{align*}
\end{example}
The BCH aids in deriving the following important existence lemma for noncommutative analogues of the first exponential identity of \cref{e:EIs} (EI for short).
\begin{lemma}\label{l:exist}
    $[B, A]$ is independent of $B$ if and only if there exists a formal series $\Psi_A : \K \to \A$ depending on $A$ and satisfying $\Psi_A(0) = 1$ and $\Psi'_A(0) = A$, such that for all $t \in \K$
    \begin{equation}\label{e:psi}
        e^{(A+B)t} = \Psi_A(t) e^{B t}.
    \end{equation}
\end{lemma}

This lemma also ensures the existence of the binomial theorem in this special case. Indeed, the formula can be retrieved after applying the Cauchy product formula and identifying the coefficients of $t^n$.

\begin{proof}
    \hfill
    \begin{itemize}
        \item[( $\imp$ )]
        We know that
        \[
         [A+B, -B] = [B, A].
        \]
        So using \Cref{l:chain}, we see that any combination of commutators of $A+B$ and $-B$ are functions of $A$.
        Thus, by the BCHD, we have
        \[
        \log(e^{(A+B)t}e^{-B t}) = (A+B)t-Bt + \ldots = At + \ldots,
        \]
        where the dots are the commutator terms of higher order in $t$ that are independent of $B$ (see \Cref{x:BCHD}).
        Thus we can write $e^{(A+B)t} = e^{At + \ldots} e^{B t}$, and setting $\Psi_A(t) = e^{A t + \ldots}$, we notice that $\Psi_A$ has the properties we need.
        \item[( $\isimp$ )]
        Taking the second derivative at $t=0$ in the equality (\ref{e:psi}), we arrive at
        \[
        (A+B)^2 = \Psi_A''(0) + 2 \Psi_A'(0) B + \Psi_A(0)B^2 = \Psi_A''(0) + 2 A B + B^2,
        \]
        and by using \cref{e:comid}, we obtain $[B, A] = \Psi_A''(0) - A^2$, which is independent of $B$.
    \end{itemize}
\end{proof}

We can express the EI of \Cref{l:exist} in various equivalent forms, which enables us to extend the formula we know (and prove in the following sections) to other similar cases.

\begin{proposition}\label{p:similar}
    \begin{enumerate}[label=\textnormal{\roman*.}]
    \item[] \hspace{-29pt} If
    {%
    \setlength{\belowdisplayskip}{-3pt}%
    \setlength{\abovedisplayskip}{3pt}%
    \begin{align*}
        [B, A] = f(A) &\implies e^{(A+B)t} = \Psi_A(t) e^{B t}
        \intertext{\hspace{-29pt} then the following are also true}    
        \intertext{\item Symmetry}
        [B, A] = f(B) &\implies e^{(A+B)t} = e^{A t} \Psi_B(t)
        \intertext{\item Antinormal EI}
        [B, A] = -f(A) &\implies e^{(A+B)t} = e^{Bt} \Psi_A(t)
        \intertext{\item Special case of the BCH}
        [B, A] = -f(A+B) &\implies e^{A t} e^{B t} = \Psi_{A+B}(t).
    \end{align*}
    }   
    \end{enumerate}
\end{proposition}

\begin{proof}
    For the symmetry, if $[B, A] = f(B)$ then $[-A-B, B] = f(B)$. Thus, according to the hypothesis
    \[
    e^{(B-A-B)t} = \Psi_B(t) e^{(-A-B)t}.
    \]
    In other words,
    \[
    e^{(A+B)t} = e^{A t} \Psi_B(t).
    \]
    The special case of the BCH is treated similarly by noticing that $[B, A] = -f(A+B)$ is the same as $[-B, A+B] = f(A+B)$.
    For the antinormal EI, we could just use the symmetry with $[A, B] = f(A)$. Alternatively, we can use the operator $\L$ of \Cref{t:L} which was designed for this purpose, and make the following computation
    \begin{align*}
        \L(e^{(A+B)t}) &= \L(\Psi_A(t) e^{Bt}) \\
        e^{\L(A+B)t} &= \L(e^{Bt}) \L(\Psi_A(t)) \\
        e^{(\L(A)+\L(B))t} &= e^{\L(B)t} \Psi_{\L(A)}(t) \\
        e^{(C+D)t} &= e^{Dt} \Psi_C(t),
    \end{align*}
    with $[C, D] = -f(D)$ and $C$ and $D$ that we can freely rename to $A$ and $B$ respectively.
\end{proof}

The following proposition is an additional general tool, for computing EIs, generalizing \Cref{l:exist}, that will prove useful later on.

\begin{proposition}\label{p:l A}
    If $[B, A]$ does not depend on $B$ then for $\Psi_A$ defined in \Cref{l:exist}, and for all $t, \lambda \in \K$, we have the following
    \[
     e^{(\lambda A + B)t} = \Psi_A(t)^\lambda e^{B t}.
    \]
\end{proposition}

\begin{proof}
    We first prove the proposition for $\lambda = n \in \N$. Because $[(n-1)A + B, A] = [B,A]$ does not depend on $B$, we have
    \[
    e^{(nA+B)t} = e^{(A+(n-1)A+B)t} = \Psi_A(t) e^{((n-1)A+B)t}.
    \]
    Repeating the process another $n-1$ times, we recover the expected term $\Psi_A(t)^n e^{B t}$.
    
    Now to generalize this to any $\lambda$, we make a calculation similar to the proof of \Cref{p:expo}, where we first use the definition of exponentiation and the result above:
    \[
    \Psi_A(t)^\lambda e^{B t} = \sum_{k=0}^\infty \binom{\lambda}{k} \sum_{j=0}^k \binom{k}{j} (-1)^{k-j} \Psi_A(t)^j e^{B t} = \sum_{k=0}^\infty \binom{\lambda}{k} \sum_{j=0}^k \binom{k}{j} (-1)^{k-j} e^{(jA + B) t}
    \]
    Let us introduce the notation $\uline{A^i B^j}$ which represents the sum of all possible product orders for $i$ $A$'s and $j$ $B$'s. For example, $\uline{A^2 B} = \uline{ABA} = \uline{BA^2} \defeq A^2 B + ABA + BA^2$ and the notation carries this property: $\uline{(\lambda A)^i (\mu B)^j} = \lambda^i \mu^j \uline{A^i B^j}$. We can now write the noncommutative expansion of a binomial as
    \[
    (A+B)^n = \sum_{k=0}^n \uline{A^k B^{n-k}},
    \]
    thus
    \begin{align*}
        \Psi_A(t)^\lambda e^{B t}
        &= \sum_{k=0}^\infty \binom{\lambda}{k} \sum_{j=0}^k \binom{k}{j} (-1)^{k-j}  \sum_{n=0}^\infty \frac{t^n}{n!} \sum_{i=0}^n \uline{(jA)^i B^{n-i}} \\
        &= \sum_{n=0}^\infty  \frac{t^n}{n!} \sum_{i=0}^n \sum_{k=0}^\infty \binom{\lambda}{k} \bra{\sum_{j=0}^k \binom{k}{j}(-1)^{k-j} j^i} \uline{A^i B^{n-i}}.
    \end{align*}
    Using the closed form of the Stirling numbers of the second kind (\ref{e:stirformula}) and then their definition (\ref{e:stir2}), we conclude the computation:
    \begin{align*}
        \Psi_A(t)^\lambda e^{B t}
        &= \sum_{n=0}^\infty  \frac{t^n}{n!} \sum_{i=0}^n \bra{\sum_{k=0}^\infty \binom{\lambda}{k}  \Stir{i}{k} k!} \uline{A^i B^{n-i}} \\
        &= \sum_{n=0}^\infty  \frac{t^n}{n!} \sum_{i=0}^n \bra{ \sum_{k=0}^i \Stir{i}{k} (\lambda)_k} \uline{A^i B^{n-i}} \\
        &= \sum_{n=0}^\infty  \frac{t^n}{n!}  \sum_{i=0}^n \lambda^i \uline{A^i B^{n-i}} = \sum_{n=0}^\infty  \frac{t^n}{n!}  (\lambda A + B)^n \\
        &= e^{(\lambda A + B)t}.
    \end{align*}
\end{proof}

\section{The exponential identity and the binomial theorem with a quadratic commutator}\label{s:quadra}

In the subsequent sections, we will establish the validity of EIs under the condition
\begin{equation}\label{e:comm}
    [B, A] = \alpha + \epsilon A - \lambda A^2,
\end{equation}
that is when the commutator is at most quadratic. We denote the algebra generated by $A$ and $B$ satisfying \cref{e:comm} with $\A(\alpha, \epsilon, -\lambda)$.

This analysis readily extends to cases where the polynomial is in $B$, using \Cref{p:similar}. While the most general formula presents a complex formulation, making an exhaustive derivation of its corresponding binomial theorem (BT for short) impractical, there exist numerous limiting cases where significant simplifications occur, often accompanied by concise BTs and where the generated algebra is well-known. The cases we will study are as follows.

\begin{enumerate}[(\alph*)]
    \item The most general quadratic case (\Cref{t:lea}) has no straightforward binomial.
    \item $\lambda = 1$ (\Cref{t:1ea}) which appears to be the only value for $\lambda$ that produces a simple binomial.
    \item $\epsilon = 0$ (\Cref{c:l0a}) which remains too complex to carry a nice binomial.
    \item $\alpha = 0$  (\Cref{c:le0}), which surprisingly does have a binomial.
    \item $\lambda = 0$ (\Cref{c:0ea}). $\A(\alpha, \epsilon, 0)$ is referred to as the \emph{generalized Ore algebra} \cite{Ore,DFG} and lacks a binomial.
    \item $\alpha = \epsilon = 0$ (\Cref{c:l00}). $\A(0, 0, -\lambda)$ is the well-known \emph{Jordan plane}, also sometimes called the \emph{meromorphic Weyl algebra} \cite{StirBessel,QSF}, with a trivial scaling.
    \item $\alpha = \lambda = 0$ (\Cref{c:0e0}). $\A(0, \epsilon, 0)$ is the well-known (scaled) \emph{shift algebra}.
    \item $\epsilon = \lambda = 0$ (\Cref{c:00a}). $\A(\alpha, 0, 0)$ is the well-known (scaled) \emph{Weyl algebra}.
\end{enumerate}

To express the first two results more clearly, we introduce, for $A \in \mathcal{A}$, a power series in $t$, continuous in all variables $t, x, y \in \mathbb{K}$:
\begin{equation}\label{e:Phi}
    \Phi_A(t, x, y) \defeq
    \begin{cases}
        \dfrac{y e^{x t} - x e^{y t}}{y - x} + A \dfrac{e^{y t} - e^{x t}}{y - x} & \com{if} x \neq y\\[8pt]
        (1-xt)e^{x t} +  A t e^{x t} & \com{if} x = y,
    \end{cases}
\end{equation}
furthermore, it will prove useful to define
\[
[x]_n \defeq \prod_{i=1}^{n-1} (1 - i x),
\]
which can be alternatively expressed, for $x \neq 0$, as $[x]_n = x^n (1/x)_n$.

\subsection{The preliminary case \texorpdfstring{$\lambda=1$}{lambda=1}}\label{s:first}

We will begin our investigation with what may appear as an arbitrary case: when $\lambda = 1$ in the equation $[B, A] = \alpha + \epsilon A - \lambda A^2$ where neither $\alpha$ nor $\epsilon$ are zero at the same time (subject of \Cref{t:lea}). Remarkably, $1$ seems to be the only value for $\lambda$ for which the formulation of the BT in \Cref{t:lea} is simple. This observation stems from the simplification of the squared term in $(A+B)A = A^2 + AB + \alpha + \epsilon A - A^2 = AB + \alpha + \epsilon A$ (as we will witness in the proof). Furthermore,  the study of this special case constitutes a necessary initial step towards addressing the more general case of \Cref{t:lea}.

\begin{theorem}\label{t:1ea}
    Let $A$ and $B$ be two elements of $\A$ and $\alpha$ and $\epsilon$ two complex numbers such that $\alpha \neq 0$ or $\epsilon \neq 0$. Let $r, \rho$ be the two roots of the polynomial $\alpha + \epsilon X - X^2$.
    
    Let $(u_n)_{n\in\N}, (v_n)_{n\in\N}$ be two complex sequences defined by
    \[
    \begin{cases}
        u_0 = 0,\, u_1 = 1 \\
        u_{n+2} = \epsilon u_{n+1} + \alpha u_n
    \end{cases}
    \qquad
    \begin{cases}
        v_0 = 1,\, v_1 = 0 \\
        v_{n+2} = \epsilon v_{n+1} + \alpha v_n,
    \end{cases}
    \]
    or equivalently by
    \[
    u_n = \begin{cases}
        \dfrac{\rho^n - r^n}{\rho-r} & \com{if} r \neq \rho \\[8pt]
        n r^{n-1} & \com{if} r = \rho
    \end{cases}
    \qquad\quad   
    v_n = \begin{cases}
        \dfrac{\rho r^n - r\rho^n}{\rho-r} & \com{if} r \neq \rho \\[8pt]
        (1-n)r^n & \com{if} r = \rho.
    \end{cases}
    \]
    Then these statements are equivalent:
    \begin{enumerate}[label=\textnormal{\roman*.}]
        \item \makebox[\linewidth]{$[B, A] = \alpha + \epsilon A - A^2$}
        \item for $n\in\N$, \[
        (A+B)^n = \sum_{k=0}^n \binom{n}{k} \bra{v_k  + u_k A} B^{n-k}
        \]
        \item for $t\in\K$, \[
        e^{(A+B)t} = \Phi_A(t, r, \rho) e^{B t}.
        \]
    \end{enumerate}
\end{theorem}

We notice this elegant structure where a single equation is equivalent to a countably infinite number of equations, which in turn holds equivalence to an uncountably infinite number of equations.

\begin{proof}
    The relationships between the coefficients and roots of a polynomial read $ \alpha = - r \rho$ and $\epsilon = r + \rho$. We can employ these alongside the explicit expressions of $u_n$ and $v_n$ to derive the relationship
    \[
    v_{n+1} = \alpha u_n \quad \et \quad u_{n+1} = v_n + \epsilon u_n. \tag{$*$}
    \]
    We proceed to prove by induction (i $\imp$ ii). Since the base case $n=0$ is satisfied, let us assume that the property holds for $n \in \mathbb{N}$. 
    \begin{align*}
        & (A+B)^{n+1} = \sum_{k=0}^n \binom{n}{k} (A+B)\bra{v_k + u_k A} B^{n-k} \\
        ={}& \sum_{k=0}^n \binom{n}{k} \bra{v_k A + u_k A^2 + v_k B + u_k (AB+ \alpha + \epsilon A - A^2)} B^{n-k} \\
        ={}& \sum_{k=0}^n \binom{n}{k} (u_k AB +  v_k B + (v_k + \epsilon u_k)A + \alpha u_k) B^{n-k},
    \end{align*}
    where we split the sum and apply $(*)$ to obtain
    \begin{align*}
        (A+B)^{n+1} 
        &= \sum_{k=0}^n \binom{n}{k} (u_k A +  v_k) B^{n+1-k} + \sum_{k=0}^n \binom{n}{k} (u_{k+1}A + v_{k+1} B) B^{n-k} \\
        &= \sum_{k=0}^{n+1} \bra{\binom{n}{k} + \binom{n}{k-1}} \bra{v_k + u_k A} B^{n+1-k},
    \end{align*}
    and finally, utilizing Pascal's identity, we conclude the induction.
    
    For the proof of (ii $\imp$ i), first we compute
    \[
    (A+B)^2 = (v_0 + u_0 A) B^2 + 2 (v_1 + u_1A)B + v_2 + u_2A = B^2 + 2 AB + \alpha + \epsilon A,
    \]
    then we make use of \cref{e:comid} to obtain $[B, A] = \alpha + \epsilon A -  A^2$, as expected.
    
    For (iii $\imp$ ii), we notice that by the explicit formula for $(u_n)_{n\in\N}$ and $(v_n)_{n\in\N}$, we precisely have
    \[
    \Phi_A(t, r, \rho) =  \sum_{n=0}^\infty \frac{t^n}{n!} (v_n + u_n A) ,
    \]
    which is a Cauchy product with $e^{Bt}$ and identification away from the desired result. Similarly, we derive (ii $\imp$ iii) by summing the binomial theorem with the factor $t^n/n!$.
\end{proof}

\subsection{The general result and its various limiting cases}\label{s:second}

\Cref{t:1ea} was a mandatory step in proving the result in full generality, since we were able to connect the commutator equation i with the EI through the BT. This would have been considerably more challenging to achieve in the general case, given the absence of straightforward formulations for the BT.

\begin{theorem}\label{t:lea}
    Let $A$ and $B$ be two elements of $\A$ and $\alpha$, $\epsilon$ and $\lambda \neq 0$ three complex numbers such that $\alpha \neq 0$ or $\epsilon \neq 0$.
    Let $r, \rho$ be two roots of the polynomial $\alpha + \epsilon X - \lambda X^2$.
    Then these statements are equivalent:
    \begin{enumerate}[label=\textnormal{\roman*.}]
        \item \makebox[\linewidth]{$[B, A] = \alpha + \epsilon A - \lambda A^2$}
        \item for $t\in\K$, \[
        e^{(A+B)t} = \Phi_A(\lambda t, r, \rho)^{1/\lambda} e^{B t}.
        \]
    \end{enumerate}
\end{theorem}

\begin{remark}
    It is worth mentioning that deriving a straightforward BT is possible, as, when $r\neq\rho$, we have the equality
    \begin{equation}
       \Phi_A(\lambda t, r, \rho)^{1/\lambda} = e^{rt} \bra{1+(A-r)\frac{e^{\lambda(\rho - r)t}-1}{\rho - r}}^{1/\lambda},
    \end{equation}
    which we can expand with \cref{e:expo,e:stir2gf}. However, after identification, the BT would involve radicals and, therefore, would not encompass the simple fact that integer coefficients in the commutator should lead to integer coefficients in the BT.
\end{remark}

\begin{proof}
    According to \Cref{p:l A}, we are aware of the existence of $\Psi_A$ such that
    \begin{equation}\label{e:psilambda}
        e^{(\lambda A + B) t} = \Psi_A(t)^\lambda e^{B t}.
    \end{equation}
    Moreover, equivalently to i, we find, through linearity, that $[B, \lambda A] = \lambda \alpha + \epsilon (\lambda A) - (\lambda A)^2$, thus according to \Cref{t:1ea}
    \[
    e^{(\lambda A + B) t} = \Phi_{\lambda A}(t, \lambda r, \lambda\rho) e^{B t},
    \]
    because $\lambda r, \lambda \rho$ are roots of the polynomial $\lambda \alpha + \epsilon X - X^2$ . We can identify the above expression with \cref{e:psilambda} and simplify it via the definition of $\Phi$ given in (\ref{e:Phi}), to obtain
    \[
    \Psi_A(t)^\lambda = \Phi_A(\lambda t, r, \rho),
    \]
    and by raising this equation to the power $1/\lambda$ (with \Cref{p:expo}), we arrive at the desired result.
\end{proof}

The proof is highly specific to the degree 2 case, and it does not appear that it can be generalized to the case of the commutator being a polynomial of higher degree. This limitation likely arises because the degree 2 cases are more natural and straightforward, while higher-degree cases may introduce additional complexities that are not easily addressed within the framework of the proof above.

\begin{corollary}\label{c:l0a}
    Let $A$ and $B$ be two elements of $\A$ and $\alpha \neq 0$ and $\lambda \neq 0$ two complex numbers.
    Then these statements are equivalent:
    \begin{enumerate}[label=\textnormal{\roman*.}]
        \item \makebox[\linewidth]{$[B, A] = \alpha - \lambda A^2$}
        \item for $t\in\K$, \[
        e^{(A+B)t} = \bra{\cosh(\sqrt{\lambda \alpha}t) + A \sqrt{\frac{\lambda}{\alpha}} \sinh(\sqrt{\lambda\alpha}t)}^{1/\lambda} e^{B t}.
        \]
    \end{enumerate}
\end{corollary}

Using the identity $(\cosh x + A \sinh x)^p = (1 - (\tanh x)^2)^{-p/2} (1 + A \tanh(x))^p$, one can derive a BT for this specific instance however, they would include coefficients of powers of the hyperbolic tangent that are not easily expressible.

\begin{proof}
    Looking at \Cref{t:lea}, with $\epsilon = 0$ we have $r, \rho = \pm \sqrt{\alpha/\lambda}$. The proof is then straightforward.
\end{proof}

\begin{corollary}\label{c:le0}
    Let $A$ and $B$ be two elements of $\A$ and $\lambda \neq 0$ and $\epsilon \neq 0$ two complex numbers.
    Then these statements are equivalent:
    \begin{enumerate}[label=\textnormal{\roman*.}]
        \item  \makebox[\linewidth]{$[B, A] = \epsilon A - \lambda A^2$}
        \item for $n\in\N$, \begin{align*}
            (A+B)^n &= \sum_{k=0}^n \binom{n}{k} \bra{\sum_{j=0}^k \Stir{k}{j} \epsilon^{k-j} [\lambda]_j A^j} B^{n-k} \\
        &= \sum_{k=0}^n \epsilon^{n-k} \sum_{j=0}^k  \binom{n}{j} \Stir{n-j}{k-j} [\lambda]_{k-j} A^{k-j} B^j
        \end{align*}
        \item for $t\in\K$, \[
        e^{(A+B)t} = \bra{1+\lambda A \frac{e^{\epsilon t}-1}{\epsilon}}^{1/\lambda} e^{B t}.
        \]
    \end{enumerate}
\end{corollary}

\begin{proof}
    Considering \Cref{t:lea}, with $\alpha = 0$, we find $r=0$ and $\rho = \epsilon/\lambda$. The proof of (i $\eq$ iii) is straightforward. For ii, we utilize the EI. First, we express the following in its power series form and then use the Stirling generating function (\ref{e:stir2gf})
    \begin{align*}
        &\bra{1+\lambda A \frac{e^{\epsilon t}-1}{\epsilon}}^{1/\lambda}
        = \sum_{j=0}^\infty \binom{1/\lambda}{j} \lambda^j A^j \bra{\frac{e^{\epsilon t}-1}{\epsilon}}^j \\
        ={}& \sum_{j=0}^\infty (1/\lambda)_j \lambda^j A^j \inv{\epsilon^j }\sum_{n=j}^\infty \Stir{n}{j} \frac{(\epsilon t)^n}{n!} = \sum_{n=0}^\infty \frac{t^n}{n!} \sum_{j=0}^n \Stir{n}{j} \epsilon^{n-j} [\lambda]_j A^j.
    \end{align*}
    By the Cauchy product and identification
    \[
    (A+B)^n = \sum_{k=0}^n \binom{n}{k} \bra{\sum_{j=0}^k \Stir{k}{j} \epsilon^{k-j} [\lambda]_j A^j} B^{n-k}.
    \]
    The second formula is obtained by summing the triangular array of terms in the orthogonal direction:
    \begin{equation}\label{e:sumarray}
        \sum_{0\leq j\leq k\leq n} c_{j,k} = \sum_{0\leq j\leq k\leq n} c_{k-j,n-j}.
    \end{equation}
    
\end{proof}

\begin{corollary}\label{c:0ea}
    Let $A$ and $B$ be two elements of $\A$ and $\epsilon \neq 0$ and $\alpha$ two complex numbers.
    Then these statements are equivalent:
    \begin{enumerate}[label=\textnormal{\roman*.}]
        \item \makebox[\linewidth]{$[B, A] = \alpha + \epsilon A$}
        \item for $t\in\K$, \[
        e^{(A+B)t} = e^{\tfrac{\alpha}{\epsilon}\bra{\tfrac{e^{\epsilon t}-1}{\epsilon} - t}} e^{A \tfrac{e^{\epsilon t}-1}{\epsilon}} e^{B t}
        \]
    \end{enumerate}
\end{corollary}

\begin{proof}
    Beginning with \Cref{t:lea} and aiming to take the limit $\lambda \to 0$, we select $|\lambda| > 0$ sufficiently small such that the discriminant $\Delta_\lambda^2 \defeq \epsilon^2 + 4 \lambda\alpha$ is nonzero. The distinct roots can then be expressed as
    \[
    r_\lambda = \frac{\epsilon-\Delta}{2\lambda} = -\frac{\alpha}{\epsilon} + o(1), \quad \rho_\lambda = \frac{\epsilon+\Delta}{2\lambda} = \frac{\epsilon}{\lambda} + o(\lambda^{-1}),
    \]
    because $\Delta_\lambda = \epsilon + 2 \lambda\alpha/\epsilon + o(\lambda)$. Utilizing $\rho_\lambda - r_\lambda = \Delta_\lambda/\lambda$, we then obtain
    \begin{align*}
        \log(\Phi_A(\lambda t, r_\lambda, \rho_\lambda)^{1/\lambda})
        &= \inv\lambda (\Phi_A(\lambda t, r_\lambda, \rho_\lambda) - 1) + o(1) \\
        &= \inv{\Delta_\lambda} (\rho_\lambda e^{r_\lambda \lambda t} - r_\lambda e^{\rho_\lambda \lambda t} + A (e^{\rho_\lambda \lambda t} - e^{r_\lambda \lambda t}) - (\rho_\lambda - r_\lambda)) + o(1) \\
        &= \inv{\Delta_\lambda} (\rho_\lambda (e^{r_\lambda \lambda t} - 1) - r_\lambda (e^{\rho_\lambda \lambda t} - 1) + A (e^{\rho_\lambda \lambda t} - e^{r_\lambda \lambda t})) + o(1) \\
        &\xrightarrow[\lambda\to0]{} \inv\epsilon \bra{-\alpha t + \frac{\alpha}{\epsilon}(e^{\epsilon t} - 1) + A (e^{\epsilon t} - 1)},
    \end{align*}
    which upon applying the exponential yields the EI of the corollary.
\end{proof}

\begin{example}\label{x:zass}
    A notable implication of the theorem is a specific instance of the BCH. By using \Cref{p:l A}, we can make the computation
    \[
    e^{\bra{\tfrac{\epsilon t}{e^{\epsilon t}-1}A + B}t} = \bra{e^{\tfrac{\alpha}{\epsilon}\bra{\tfrac{e^{\epsilon t}-1}{\epsilon} - t}} e^{A \tfrac{e^{\epsilon t}-1}{\epsilon}}}^{\tfrac{\epsilon t}{e^{\epsilon t}-1}} e^{B t} = e^{\tfrac{\alpha t}{\epsilon}\bra{1 - \tfrac{\epsilon t}{e^{\epsilon t}-1}}} e^{A t} e^{B t},
    \]
    which can also be expressed as
    \begin{equation}
        e^{At} e^{Bt} = e^{\bra{A+B + [B, A] \bra{\tfrac{t}{e^{\epsilon t}-1}-\tinv{\epsilon}}}t}.
    \end{equation}
    This formula is a special case of the bivariate $[B, A] = \alpha + \epsilon A + \sigma B$, for which a closed form of the BCH was established in \cite{BCHclosed}.
\end{example}

In the rest of this section, the results we present are already known. For example, in this instance, known as the meromorphic Weyl algebra, the EI is commonly referred to as the ``\emph{Berry identity}'' \cite[Appendix (a)]{Berry}. Viskov \cite{Viskov98} first expanded it into the BT, but it was also notably independently found by Benaoum \cite{Benaoum98}, along with its $q$-analogue \cite{Benaoum99}.

\begin{corollary}[The Berry identity]\label{c:l00}
    Let $A$ and $B$ be two elements of $\A$ and $\lambda \neq 0$ a complex number.
    Then these statements are equivalent:
    \begin{enumerate}[label=\textnormal{\roman*.}]
        \item \makebox[\linewidth]{$[B, A] = -\lambda A^2$}
        \item for $n\in\N$, \[
        (A+B)^n = \sum_{k=0}^n  \binom{n}{k} [\lambda]_k A^k B^{n-k}
        \]
        \item for $t\in\K$, \[
        e^{(A+B)t} = \bra{1+\lambda At }^{1/\lambda} e^{B t}.
        \]
    \end{enumerate}
\end{corollary}

\begin{proof}
    Using \Cref{c:le0} with $\epsilon \to 0$ and the limit
    \[
    \lim_{\epsilon \to 0} \frac{e^{\epsilon t}-1}{\epsilon} = t.
    \]
\end{proof}

The EI in the shift algebra is known as \emph{Sack's identity} \cite{Sack} and Viskov \cite{Viskov95} later expressed the BT by expanding Sack's identity.

\begin{corollary}[Sack's identity]\label{c:0e0}
    Let $A$ and $B$ be two elements of $\A$ and $\epsilon \neq 0$ a complex number.
    Then these statements are equivalent:
    \begin{enumerate}[label=\textnormal{\roman*.}]
        \item \makebox[\linewidth]{$[B, A] = \epsilon A$}
        \item for $n\in\N$, \begin{align*}
            (A+B)^n &= \sum_{k=0}^n \binom{n}{k} \bra{\sum_{j=0}^k \Stir{k}{j} \epsilon^{k-j} A^j} B^{n-k} \\
        &= \sum_{k=0}^n \epsilon^{n-k} \sum_{j=0}^k  \binom{n}{j} \Stir{n-j}{k-j} A^{k-j} B^j
        \end{align*}
        \item for $t\in\K$, \[
        e^{(A+B)t} = e^{A \tfrac{e^{\epsilon t}-1}{\epsilon}} e^{B t}.
        \]
    \end{enumerate}
\end{corollary}

\begin{proof}
    Using \Cref{c:le0} for ii with $\lambda = 0$ and \Cref{c:0ea} for iii with $\alpha = 0$.
\end{proof}

The subsequent result is arguably one of the most renowned, as it encompasses the ``\emph{Glauber formula}'', which is a well-known and useful result in quantum mechanics, as the underlying mathematical structure is often the Weyl algebra. Although the binomial theorem is a direct consequence of the Glauber formula, it seems that it was first formulated by Yamazaki \cite{Yamazaki}.

\begin{corollary}[The Glauber formula]\label{c:00a}
    Let $A$ and $B$ be two elements of $\A$ and $\alpha$ a complex number.
    Then these statements are equivalent:
    \begin{enumerate}[label=\textnormal{\roman*.}]
        \item \makebox[\linewidth]{$[B, A] = \alpha$}
        \item for $n\in\N$, \[
        (A+B)^n = \sum_{i=0}^{\floor{n/2}} \frac{n! (\alpha/2)^i}{i!(n-2i)!} \sum_{k=0}^{n-2i} \binom{n-2i}{k} A^k B^{n-2i-k}
        \]
        \item for $t\in\K$, \[
        e^{(A+B)t} = e^{\alpha \tfrac{t^2}{2}} e^{A t} e^{B t}.
        \]
    \end{enumerate}
\end{corollary}

\begin{proof}
    Using \Cref{c:0ea} for iii with $\epsilon \to 0$ and the limits
    \[
    \lim_{\epsilon \to 0} \frac{e^{\epsilon t}-1}{\epsilon} = t, \qquad \lim_{\epsilon \to 0} \inv{\epsilon} \bra{\frac{e^{\epsilon t}-1}{\epsilon}-t} = \frac{t^2}{2}.
    \]
    The BT follows from expanding the EI. We apply the Cauchy product to the following power series
    \[
    e^{\alpha \tfrac{t^2}{2}} = \sum_{p=0}^\infty t^p \frac{(\alpha/2)^{p/2}}{(p/2)!} [p \text{ is even}], \qquad e^{A t} e^{B t} = \sum_{n=0}^\infty \frac{t^n}{n!} \sum_{k=0}^n \binom{n}{k} A^k B^{n-k},
    \]
    (where Iverson's brackets $[p \text{ is even}]$ take the value 1 if $p$ is even, and 0 if $p$ is odd) to obtain
    \begin{align*}
        e^{\alpha \tfrac{t^2}{2}} e^{A t} e^{B t}
        &= \sum_{n=0}^\infty t^n \sum_{j=0}^n \frac{(\alpha/2)^{j/2}}{(j/2)!(n-j)!} [j \text{ is even}] \sum_{k=0}^{n-j} \binom{n-j}{k} A^k B^{n-j-k} \\
        &= \sum_{n=0}^\infty \frac{t^n}{n!} \sum_{i=0}^{\floor{n/2}} \frac{n! (\alpha/2)^i}{i!(n-2i)!} \sum_{k=0}^{n-2i} \binom{n-2i}{k} A^k B^{n-2i-k}.
    \end{align*}
    Hence by identification with $e^{(A+B)t}$, we prove ii.
\end{proof}

\section{The case of a bivariate commutator}\label{s:bivariate}

The existence of an EI and a BT is guaranteed when the commutator is univariate, as established in \Cref{l:exist} and \Cref{p:similar} i.  However, this guarantee is lost in the bivariate case, as noted in \cite{Rosengren}. Nonetheless, as hinted at by \Cref{p:similar} iii, the BCH formula and its less renowned counterpart, the \emph{Zassenhaus formula}\footnote{which, in short, provides an expression for $e^{-A}e^{-B}e^{A+B}$}, appear to be more promising paths for investigating bivariate commutators. In particular, closed forms for both the BCH \cite{BCHclosed,BCHLie} and the Zassenhaus formula \cite{Zassclosed} have been derived in the case $[A, B] = \alpha + \epsilon A + \sigma B$.

Additional formulas for the BCH in bivariate settings can be derived using \Cref{p:similar} iii and the results from \Cref{s:quadra}. One such formula is presented in the following theorem.

\begin{theorem}\label{t:A+B}
    Let $A$ and $B$ be two elements of $\A$ and $\epsilon$ and $\lambda \neq 0$ two complex numbers.
    Then these statements are equivalent:
    \begin{enumerate}[label=\textnormal{\roman*.}]
        \item \makebox[\linewidth]{$[B, A] = \epsilon (A+B) + \lambda (A+B)^2$}
        \item for $n\in\N$, \[
        [\lambda]_n (A+B)^n = \sum_{k=0}^n \coeff{n}{k}  \epsilon^{n-k} \sum_{j=0}^k \binom{k}{j} A^j B^{k-j} \\
        \]
        \item for $t\in\K$, \[
        e^{At} e^{B t} = \bra{1+\lambda (A+B) \frac{1-e^{-\epsilon t}}{\epsilon}}^{1/\lambda}.
        \]
    \end{enumerate}
\end{theorem}

\begin{proof}
    Our goal is to translate the equivalent statements of \Cref{c:le0} into this context. From i, iii follows with the special case of the BCH given by \Cref{p:similar}, applied to the EI of \Cref{c:le0}, where $\epsilon$ should be replaced with $-\epsilon$. For ii, we first rewrite iii as
    \[
    e^{A \tfrac{-\log(1-\epsilon t)}{\epsilon}} e^{B \tfrac{-\log(1-\epsilon t)}{\epsilon}} = (1+\lambda (A+B)t)^{1/\lambda}.
    \]
    On the left side, we use \cref{e:stir1gf} after applying the Cauchy product
    
    \begin{align*}
        &e^{A \tfrac{-\log(1-\epsilon t)}{\epsilon}} e^{B \tfrac{-\log(1-\epsilon t)}{\epsilon}} = \sum_{k=0}^\infty \frac{(-\log(1-\epsilon t))^k}{\epsilon^k k!} \sum_{j=0}^k \binom{k}{j} A^j B^{k-j} \\
        ={}& \sum_{k=0}^\infty \sum_{n=k}^\infty \coeff{n}{k} \frac{t^n}{n!} \epsilon^{n-k} \sum_{j=0}^k \binom{k}{j} A^j B^{k-j} = \sum_{n=0}^\infty \frac{t^n}{n!} \sum_{k=0}^n \coeff{n}{k}  \epsilon^{n-k} \sum_{j=0}^k \binom{k}{j} A^j B^{k-j},
    \end{align*}
    and the right-hand side becomes, after a calculation done earlier:
    \[
    (1+\lambda (A+B)t)^{1/\lambda} = \sum_{n=0}^\infty \frac{t^n}{n!} (A+B)^n [\lambda]_n,
    \]
    which means that by identification, we finally obtain
    \[
    [\lambda]_n (A+B)^n = \sum_{k=0}^n \coeff{n}{k}  \epsilon^{n-k} \sum_{j=0}^k \binom{k}{j} A^j B^{k-j}.
    \]
    The steps can be made backward, effectively showing the equivalence between ii and iii.
\end{proof}

Here, we observe an interesting phenomenon that may appear contradictory: When $[B, A]$ contains degree-two terms in both $A$ and $B$, one may encounter a binomial theorem with rational coefficients (after dividing by $[\lambda]_n$ in ii), even when all parameters are integers.

\begin{example}
    To gain a better understanding of this occurrence, we aim to illustrate a special case proposed in \cite{Rosengren}, namely, the case $[B, A] = \lambda A^2 + \mu B^2$. The reader may wish to attempt on their own to normal order the term $B^2 A$, and might stop when arriving at
    \[
    B^2 A = 2 \lambda^2 A^3 + 2\lambda A^2 B + (1+\lambda\mu)A B^2 + 2 \mu B^3 + \lambda \mu B^2 A,
    \]
    because of the ``recursive'' aspect in $B^2 A$ of the formula. Thus, as noticed in \cite{Rosengren}, the normal ordered form of $B^2 A$ exists if and only if $\lambda \mu \neq 1$, and it is the following:
    \[
    B^2 A = \frac{2 \lambda^2}{1-\lambda\mu} A^3 + \frac{2\lambda}{1-\lambda\mu} A^2 B + \frac{1+\lambda\mu}{1-\lambda\mu}A B^2 + \frac{2 \mu}{1-\lambda\mu} B^3.
    \]
\end{example}

Note that the analogue of ``$\lambda\mu \neq 1$'' in the case of \Cref{t:A+B} is ``$\lambda$ is not a root of $[X]_n$''. Therefore, in the limiting case $\lambda\to0$, the binomial always exists and is given below.

\begin{corollary}\label{c:A+B}
    Let $A$ and $B$ be two elements of $\A$ and $\epsilon \neq 0$ a complex number.
    Then these statements are equivalent:
    \begin{enumerate}[label=\textnormal{\roman*.}]
        \item \makebox[\linewidth]{$[B, A] = \epsilon (A+B)$}
        \item for $n\in\N$, \[
        (A+B)^n = \sum_{k=0}^n \coeff{n}{k}  \epsilon^{n-k} \sum_{j=0}^k \binom{k}{j} A^j B^{k-j} \\
        \]
        \item for $t\in\K$, \[
        e^{At} e^{B t} = e^{(A+B) \tfrac{1-e^{-\epsilon t}}{\epsilon}}.
        \]
    \end{enumerate}
\end{corollary}
In this scenario, $A$ and $B$ would serve as generators of the (scaled) \emph{excedance algebra} \cite{excedance}. Additionally, alongside \Cref{x:zass}, \Cref{c:A+B} represents another special case of the result presented in \cite{BCHclosed}. This case can be readily transformed into an EI or even a Zassenhaus formula, which would then constitute a special case of \cite{Zassclosed}.

\section{A theorem from Viskov}\label{s:Viskov}

In \cite[Theorem 8.53]{book}, a special case of a theorem from Viskov \cite[equation (4)]{Viskov97} is briefly mentioned. However, it appears that the broader applicability of the more general theorem was not recognized within the context of the theory being discussed. Therefore, we aim to present a more general version of this theorem by Viskov. This theorem is remarkably powerful, to the extent that it implies \Cref{l:exist} and the theorems outlined in \Cref{s:quadra}.

Regarding the history of this theorem, Viskov initially formulated a less general but multidimensional result in \cite{Viskov91}. Subsequently, in \cite{Viskov98}, Viskov extended the unidimensional case of his result and utilized a special instance of it to reaffirm the validity of Sack's identity \cite{Viskov95}. A further generalization was presented in \cite{Viskov97}\footnote{Although \cite{Viskov95,Viskov97} were published before \cite{Viskov98}, the manuscript of the latter was received prior to the former.}. Below, we present the version from \cite{Viskov98}, as the additional generality from \cite{Viskov97} is not pertinent.

It is worth noting that Viskov's proof in \cite{Viskov97}, which broadly follows the outline provided in \cite{Viskov98}, contains a critical error that we have rectified.

\begin{theorem}[Viskov's theorem]\label{t:Viskov}
    Let $p, f, g$ be formal power series and $A, B \in \A$ such that $[B, A] = p(A)$. Then
    \[
    \exp((f(A) B + g(A))t) = e^{\gamma(t)} (e^{\phi(t)})^B,
    \]
    with
    \[
    \phi(t) = \int_0^t f(\alpha(\tau)) \odif\tau, \qquad \gamma(t) = \int_0^t g(\alpha(\tau)) \odif\tau,
    \]
    where $\alpha$ is the solution to the Cauchy problem
    \[
    \odv{\alpha}{t} = p(\alpha) f(\alpha), \quad \alpha(0) = A.
    \]
\end{theorem}
\begin{proof}
    First, we will demonstrate the special case $p(x) = 1$ and then recover the full theorem through specific substitutions. With $[B, A] = p(A) = 1$, we choose the representatives $B = D$ and $A = \x$, where $D$ is the derivative operator and $\x$ is the linear operator of multiplication by the variable $x$. For a polynomial $q$, we consider the expression
    \[
    u = u(t, x) \defeq \exp\bra{\bra{f(\x)D + g(\x)}t} q(x).
    \]
    $u$ is obviously the unique solution to the Cauchy problem
    \begin{equation}\label{e:pde}
        \pdv{u}{t} = f(x) \pdv{u}{x} + g(x) u, \quad u(0, x) = q(x).
    \end{equation}
    We will show that $v = v(t, x) \defeq e^{\gamma(t, x)} q(x+\phi(t, x))$ also satisfy \cref{e:pde}, which will yield, by unicity, equality with $u$. As in the theorem, $\phi$ and $\gamma$ are defined by
    \begin{equation}\label{e:phi}
        \phi(t, x) \defeq \int_0^t f(\alpha(\tau, x)) \odif\tau,
    \end{equation}
    and
    \begin{equation}\label{e:gamma}
        \gamma(t, x) \defeq \int_0^t g(\alpha(\tau, x)) \odif\tau,
    \end{equation}
    with $\alpha$ being the unique solution to this Cauchy problem
    \begin{equation}\label{e:alpha}
        \pdv{\alpha}{t} = f(\alpha), \quad \alpha(0, x) = x.
    \end{equation}
    Immediately with (\ref{e:phi}) and (\ref{e:alpha}), we have
    \begin{equation}\label{e:mu}
        \phi(t, x) = \int_0^t \pdv{\alpha}{t}(\tau, x) \odif\tau = \alpha(t, x) - x.
    \end{equation}
    This simplifies $v$ to $e^{\gamma(t, x)} q(\alpha(t, x))$. Now we will derive the partial differential equation satisfied by $\gamma$ and $\alpha$. \Cref{e:alpha} implies
    \begin{equation}\label{e:t}
        t = \int_0^t \pdv{\alpha}{t} \frac{\odif\tau}{f(\alpha)},
    \end{equation}
    and in \cref{e:t,e:gamma}, we make the substitution $z = \alpha$. The differential is given by $\odif z = \pdv{\alpha}{\tau} \odif\tau = f(\alpha) \odif\tau = f(z) \odif\tau$, which respectively yields
    \division{
    \[
    t = \int_x^{\alpha(t, x)} \frac{\odif z}{f(z)},
    \]
    }{
    \[
    \gamma(t, x) = \int_x^{\alpha(t, x)} \frac{g(z)}{f(z)} \odif z.
    \]
    }
    Now we differentiate with respect to $x$
    \division{
    \[
    0 = \pdv{\alpha}{x} \inv{f(\alpha)} - \inv{f(x)}, \tag{$*$}
    \]
    which we rewrite with \cref{e:alpha} as
    \[
    \pdv{\alpha}{t} = f(x) \pdv{\alpha}{x}.
    \]
    }{
    \[
    \pdv{\gamma}{x} = \pdv{\alpha}{x} \frac{g(\alpha)}{f(\alpha)} - \frac{g(x)}{f(x)},
    \]
    then we use \cref{e:gamma} and ($*$) to obtain
    \[
    \pdv{\gamma}{t} = f(x) \pdv{\gamma}{x} + g(x).
    \]
    }
    These equations suffice to show that $v = e^{\gamma(t, x)} q(\alpha(t, x))$ does satisfy \cref{e:pde}. Therefore, by the unicity of the solution to the Cauchy problem, we have
    \[
    \exp((f(\x)u + g(\x))t) q(x) = e^{\gamma(t, x)} q(x+\phi(t, x)) = e^{\gamma(t, \x)} (e^{\phi(t, \x)})^D q(x),
    \]
    where we have used the notation of \Cref{x:shift}. The above relation holds for all polynomials $q$, hence the equality of the operators
    \begin{equation}\label{e:opeq}
        \exp((f(\x)u + g(\x))t) =  e^{\gamma(t, \x)} (e^{\phi(t, \x)})^D.
    \end{equation}
    where we have rewritten the right-hand side according to \cref{e:shift}. In \cref{e:opeq}, we can replace $(\x, D)$ by $(A, B)$, as long as $[B, A] = 1$. To generalize the result for $[B, A] = p(A)$, we define $\pi$ by $\pi'(t) = 1/p(t)$, to have, by \Cref{l:chain}, $[B, \pi(A)] = 1$, thus we can apply the less general result proved just before. We will do so by replacing $(f, g, A)$ by $(f \circ \pi^{-1}, g \circ \pi^{-1}, \pi(A))$. By choosing wisely the constant term of $\pi$, we can make $\pi^{-1}$ a well-defined power series. If we set $\alpha_0 = \pi^{-1} \circ \alpha$, \cref{e:gamma,e:phi} become
    \[
    \gamma(t, x) = \int_0^t g(\alpha_0(\tau,x)) \odif\tau, \qquad \phi(t, x) = \int_0^t f(\alpha_0(\tau,x)) \odif\tau.
    \]
    $\alpha_0(0) = \pi^{-1}(\alpha(0)) = \pi^{-1}(\pi(A)) =  A$, and by the chain rule and \cref{e:alpha}
    \[
    \odv{\alpha_0}{t} = (\pi^{-1})'(\alpha) \odv{\alpha}{t} = \inv{\pi'(\pi^{-1}(\alpha))} (f \circ \pi^{-1})(\alpha) = p(\alpha_0) f(\alpha_0),
    \]
    which concludes the proof.
\end{proof}

We now provide the most important examples of the theorem. Identification of the coefficients would yield, respectively, the BT and the expansion of $(AB)^n$.
\division{
\begin{corollary}\label{c:V1}
    If $[B, A] = p(A)$ then
    \[
    e^{(A+B)t} = \exp\bra{\int_0^t \alpha(\tau) \odif\tau} e^{B t},
    \]
    where $\alpha$ is the solution to the Cauchy problem
    \[
    \odv{\alpha}{t} = p(\alpha), \quad \alpha(0) = A.
    \]
\end{corollary}
}{
\begin{corollary}\label{c:V2}
    If $[B, A] = p(A)$ then
    \[
    e^{ABt} = \exp\bra{\int_0^t \alpha(\tau) \odif\tau}^B,
    \]
    where $\alpha$ is the solution to the Cauchy problem
    \[
    \odv{\alpha}{t} = \alpha p(\alpha), \quad \alpha(0) = A.
    \]
\end{corollary}
}

The antinormal form of the theorem was also given by Viskov \cite{Viskov97}. 
\begin{theorem}
    Let $p, f, g$ be formal power series and $A, B \in \A$ such that $[B, A] = p(B)$. Then
    \[
    \exp((Af(B) + g(B))t) = (e^A)^{\phi(t)} e^{\gamma(t)},
    \]
    with
    \[
    \phi(t) = \int_0^t f(\beta(\tau)) \odif\tau, \qquad \gamma(t) = \int_0^t g(\beta(\tau)) \odif\tau,
    \]
    where $\beta$ is the solution to the Cauchy problem
    \[
    \odv{\beta}{t} = p(\beta) f(\beta), \quad \beta(0) = B.
    \]
\end{theorem}

Viskov \cite{Viskov97} claimed that the proof was similar, though this assertion does not immediately appear evident and even seems rather untrue. Nevertheless, with the tools we built, this theorem now emerges as a mere consequence of the preceding ones, sparing us the need to prove the entire theorem anew.

\begin{proof}
    We begin by applying the properties of the operator $\L$ defined in \Cref{t:L} to the equality stated in \Cref{t:Viskov}. On the left-hand side, we have
    \[
    \L(\exp((f(A)B + g(A))t)) = \exp(\L(f(A)B + g(A))t) = \exp(C f(D) + g(D))t),
    \]
    (still with $[D, C] = p(D)$) while on the right-hand side
    \[
    \L(e^{\gamma(t)} (e^{\phi(t)})^B) = \L((e^{\phi(t)})^B) \L(e^{\gamma(t)}) = (e^D)^{\L(\phi(t))} e^{\L(\gamma(t))}.
    \]
    Hence, the equality
    \[
    \exp(C f(D) + g(D))t) = (e^D)^{\L(\phi(t))} e^{\L(\gamma(t))},
    \]
    which is essentially a renaming of the variables of the original statement.
\end{proof}

And once again, the important instances are given below.
\division{
\begin{corollary}
    If $[B, A] = p(B)$ then
    \[
    e^{(A+B)t} = e^{A t} \exp\bra{\int_0^t \beta(\tau) \odif\tau},
    \]
    where $\beta$ is the solution to the Cauchy problem
    \[
    \odv{\beta}{t} = p(\beta), \quad \beta(0) = B.
    \]
\end{corollary}
}{
\begin{corollary}
    If $[B, A] = p(B)$ then
    \[
    e^{ABt} = (e^A)^{{\textstyle\int}_{\mkern-6mu0}^t \beta(\tau) \odif\tau},
    \]
    where $\beta$ is the solution to the Cauchy problem
    \[
    \odv{\beta}{t} = \beta p(\beta), \quad \beta(0) = B.
    \]
\end{corollary}
}

\section{Commutator of the form of a monomial}\label{s:monomial}

In this section, we explore the normal ordered form of $(AB)^n$ and $(A+B)^n$ when the commutator is a monomial in $A$. Although the former has been previously examined in the case where $[B, A] = h A^s$, we will now demonstrate a closely related formula for the binomial theorem when $[B, A] = h A^{s+1}$, utilizing Viskov's theorem.

\subsection{Normal ordered form of \texorpdfstring{$(AB)^n$}{(AB)\string^n}}
Mansour, Schork and Shattuck proved in \cite{GenStir} that if $A$ and $B$ are generators of the \emph{generalized Weyl algebra} \cite{book}, that is, $[B, A] = h A^s$, then we could express
\begin{equation}\label{e:MSS}
    (AB)^n = \sum_{k=0}^n \St_s(n,k) h^{n-k} A^{s(n-k)+k} B^k,
\end{equation}
where $\St_s$ represents the \emph{generalized Stirling number}, defined as follows:
\begin{equation}
    \St_0(n,k) = \Stir{n}{k}, \qquad \St_1(n,k) = \coeff{n}{k},
\end{equation}
and for $s \in \K \setminus \{0, 1\}$
\begin{align}
    \St_s(n,k)
    &= \frac{n! s^n}{k! (1-s)^k} \sum_{j=0}^k (-1)^{k-j} \binom{k}{j} \binom{n + j(1/s -1) - 1}{n} \\
    &= \inv{k!} \sum_{j=0}^k (-1)^{k-j} \binom{k}{j}\sum_{i=0}^n \coeff{n}{i} \frac{s^{n-i}j^i}{(1-s)^{k-i}}. \label{e:miss}
\end{align}
It appears that they overlooked the fact that by exchanging the sums in \cref{e:miss} and utilizing \cref{e:stirformula}, we could derive the improved closed form below, without any degenerate cases
\begin{equation}
    \St_s(n,k) = \sum_{j=k}^n \coeff{n}{j} \Stir{j}{k} s^{n-j} (1-s)^{j-k}.
\end{equation}
Their proof employed the generating function
\begin{equation}\label{e:stirgengf}
\sum_{n=k}^\infty \St_s(n,k) \frac{t^n}{n!} = \frac{\mathfrak F_s(t)^k}{k!}, 
\end{equation}
where
\begin{equation}
    \mathfrak F_s(t) \defeq \begin{cases}
        e^t-1 &\com{if} s=0 \\
    -\log(1-t) &\com{if} s=1 \\[3pt]
    \dfrac{1-(1-st)^{1-1/s}}{s-1} &\com{else.}
    \end{cases}
\end{equation}

We slightly adjusted the notation for $\St_s(n,k)$ to extract the dependency on $h$, and adjusted the lower index of the sum in \cref{e:MSS}. We set it to $0$ instead of $1$ because $\St_s(n,0) = \delta_n$. It is generally good practice to let the lower index be $0$ to also include the trivial case $n=0$.

\begin{example}\label{x:Lah}
    The obvious examples are the cases $s = 0$ or $1$, but another interesting one is the case $s = 1/2$. In fact, the formula for $\St_{1/2}$ simplifies into \emph{Lah numbers} $\Lah{n}{k} \defeq \binom{n-1}{k-1} \frac{n!}{k!}$ \cite[p.~156]{Comtet}, indeed
    \begin{equation}
        \St_{1/2}(n,k) =  2^{k-n} \sum_{j=k}^n \coeff{n}{j} \Stir{j}{k} = 2^{k-n} \Lah{n}{k},
    \end{equation}
    where we used an identity that can be easily proved by composing the Stirling generating functions (\cref{e:stir1gf,e:stir2gf}). We can rewrite \cref{e:MSS} with $[B, A] = 2 h \sqrt A$ into
    \begin{equation}
        (AB)^n = \sum_{k=0}^n \Lah{n}{k} h^{n-k} A^{\tfrac{n+k}{2}} B^k.
    \end{equation}
\end{example}

The exponential version of the theorem proved by Mansour, Schork and Shattuck is the following.
\begin{corollary}\label{c:expMSS}
    If $[B, A] = h A^s$ then
    \[
    e^{ABt} = \exp\bra{\frac{\mathfrak F_s(hA^st)}{hA^{s-1}}}^B.
    \]
\end{corollary}

\begin{proof}
    By \cref{e:MSS}
    \begin{align*}
        e^{ABt}
        &= \sum_{n=0}^\infty \frac{t^n}{n!} (AB)^n = \sum_{n=0}^\infty \frac{t^n}{n!} \sum_{k=0}^n \St_s(n,k) h^{n-k} A^{s(n-k)+k} B^k \\
        &= \sum_{k=0}^\infty \bra{\sum_{n=k}^\infty \frac{t^n}{n!} \St_s(n,k) h^n A^{sn}} h^{-k} A^{(1-s)k} B^k,
    \end{align*}
    now we use \cref{e:stirgengf} and \Cref{p:e^A^B}
    \[
    e^{ABt} =  \sum_{k=0}^\infty \frac{\mathfrak F_s(hA^st)^k}{k!} h^{-k} A^{(1-s)k} B^k = \exp\bra{\frac{\mathfrak F_s(hA^st)}{hA^{s-1}}}^B.
    \]
\end{proof}

\subsection{Binomial theorem}

The EI was found in \cite[Corollary 8.57]{book} with \Cref{c:V1}, but the connection with generalized Stirling numbers was missed.

\begin{theorem}\label{t:As+1}
    Let $A$ and $B$ be two elements of $\A$ and $h$ and $s$ two complex numbers. Then these statements are equivalent:
    \begin{enumerate}[label=\textnormal{\roman*.}]
        \item \makebox[\linewidth]{$[B, A] = h A^{s+1}$}
        \item for $n \in \N$, \begin{align*}
            (A+B)^n
            &= \sum_{k=0}^n \binom{n}{k} \bra{\sum_{j=0}^k \St_s(k,j) h^{k-j} A^{s(k-j)+j}}B^{n-k} \\
            &= \sum_{k=0}^n h^{n-k} \sum_{j=0}^k \binom{n}{j} \St_s(n-j, k-j) A^{s(n-k)+k-j} B^j
        \end{align*}
        \item for $t \in \K$, \[
        e^{(A+B)t} =  \exp\bra{\frac{\mathfrak F_s(hA^st)}{hA^{s-1}}} e^{Bt}.
        \]
    \end{enumerate}
\end{theorem}

\begin{proof}
    Suppose that $[B, A] = h A^s$. Then, by \Cref{c:V2}
    \[
    e^{ABt} = \exp\bra{\int_0^t \alpha(\tau) \odif\tau}^B,
    \]
    with
    \[
    \odv{\alpha}{t} = h \alpha^{s+1}, \quad \alpha(0) = A. \tag{$*$}
    \]
    But by \Cref{c:expMSS}, we can identify
    \begin{equation}\label{e:identify}
        \int_0^t \alpha(\tau) \odif\tau = \frac{\mathfrak F_s(hA^st)}{hA^{s-1}}.
    \end{equation}
    Now suppose i, that is, $[B, A] = h A^{s+1}$, then by \Cref{c:V1}
    \[
    e^{(A+B)t} = \exp\bra{\int_0^t \alpha(\tau) \odif\tau} e^{Bt},
    \]
    with $\alpha$ still defined by $(*)$. Thus, we can make use of \cref{e:identify} to obtain iii. The first formula of ii follows from expanding the EI and using \cref{e:sumarray}, we arrive at the second formula. To close the cycle of equivalences, knowing ii, we can simply use \cref{e:comid} to recover i.
\end{proof}

We could take advantage of the same trick used in the proof to derive, for example, one of the following formulas using the other:
\begin{itemize}
    \item The normal ordering of $(AB)^n$, when $A$ and $B$ are generators of the generalized Ore algebra, i.e., when $[B, A] = \alpha + \epsilon A$ (found in \cite{Ore})
    \item The BT when $[B, A] = A(\alpha + \epsilon A)$ (given by \Cref{c:le0})
\end{itemize}

The formulas for $\frac{\mathfrak F_s(hA^st)}{hA^{s-1}}$ in the cases $s = -1$, $0$ or $1$ are, respectively, the divisionless $At + h \frac{t^2}{2}$, $A \frac{e^{ht}-1}{h}$, and $-\log(1-hAt)$, while the other cases still involve division in their closed forms. Consequently, it is intriguing to observe that if $s+1 \in \N \setminus \{0, 1, 2\}$, then the exponential identity cannot be expressed solely in terms of elementary functions without division (although they can in a power series form). This observation sheds light on the discussion in \Cref{s:second}, where we discuss how if the degree of the commutator is not $0$, $1$ or $2$, then finding the exponential identity becomes more challenging. Thus, one of the reasons for this difficulty might be related to the appearance of division by a power of $A$.

\begin{example}
    If $[B, A] = h A^3$, then the EI (with division) is given by
    \[
    e^{(A+B)t} = e^{\tfrac{1-\sqrt{1-2 h A^2 t}}{h A}} e^{B t}.
    \]
    This case is associated with \emph{Bessel polynomials} \cite{StirBessel}. Indeed, according to Krall and Frink \cite{Bessel}, Bessel polynomials $\{y_n\}_{n\in\N}$, defined by
    \[
    y_n(x) \defeq \sum_{k=0}^n \frac{(n+k)!}{(n-k)!k!}\bra{\frac{x}{2}}^k = \sum_{k=0}^n \binom{n+k}{2k}(2k-1)!! \, x^k,
    \]
    satisfy this generating function formula
    \[
    \sum_{n=0}^\infty \frac{t^n}{n!} y_{n-1}(x)  = e^{\tfrac{1-\sqrt{1-2xt}}{x}}.
    \]
    Hence, the alternate expression for the EI (which requires no division):
    \[
    e^{(A+B)t} = \sum_{n=0}^\infty \frac{(At)^n}{n!} y_{n-1}(h A) e^{Bt}.
    \]
    Furthermore, the binomial theorem can be written as
    \[
    (A+B)^n = \sum_{k=0}^n \binom{n}{k} \bra{\sum_{j=0}^k \binom{k-1+j}{2j} (2j-1)!! \, h^j A^{k+j}} B^{n-k},
    \]
    which is at a change of index away from \Cref{t:As+1}, indicating that we have found the following closed form:
    \begin{equation}
        \St_2(n,n-k) = \binom{n-1+k}{2k}(2k-1)!!.
    \end{equation}
\end{example}

\section{Conclusion}

In our research, we have developed exponentiation and general tools for exploring noncommutative algebra. This approach enabled us to derive exponential identities and certain binomial theorems in cases where the commutator $[B, A] = f(A)$ takes the form of a polynomial of degree two or less. We also delved into the potential absence of a normal ordered form when the commutator involves two variables, providing illustrative examples to support our discussion.

We have shed light on an earlier theorem by Viskov \cite{Viskov97}, which serves as a powerful tool for solving numerous problems in normal ordering. The antinormal form was proved via the new operator $\L$. Viskov's theorem facilitated a connection between the exponential identity in cases where $f$ is a monomial of arbitrary degree and the normal ordered form of $(AB)^n$ when $f$ is one degree less, which was already documented \cite{GenStir, GenStir2, book}. This connection enabled the discovery of analogous formulas for both the exponential identity and the binomial theorem.

Further generalizing these findings by increasing the degree of the arbitrary polynomial $f$ appears to be a challenging task, as discussed. Furthermore, exploring the case of a bivariate commutator remains an intriguing yet somewhat enigmatic area of study. Extending the work outlined in \Cref{s:bivariate} and in previous literature \cite{Rosengren,Levandovskyy,BCHclosed,BCHLie,Zassclosed} could provide valuable insights into noncommutative algebra.

Moreover, our definition of a noncommutative exponentiation and our analysis of its properties in \Cref{s:expo} suggests an analogous study in noncommutative cases that should be related to some of the noncommutative exponential identities that were derived here and in the past.

Finally, now that we better understand regular normal ordering, investigating $q$-analogues seems to be another avenue of exploration. Investigating binomial theorems under conditions $BA - qAB = f(A)$, as studied in \cite{Benaoum99,qanalogue,book,Levandovskyy}, could offer further depth to our understanding of $q$-calculus and noncommutative algebra.

\subsubsection*{Acknowledgement}

I would like to thank M. Schork for his precious advice and for sharing one of the references, as well as the anonymous referees that helped improve the quality of this paper.

\bibliographystyle{plain}
\bibliography{refs}

\end{document}